\documentclass[11pt,a4paper]{article}

\usepackage{amsmath,amsthm,amssymb}
\usepackage[all]{xy}

\newtheorem{theo}{Theorem}[section]

\newtheorem{coro}[theo]{Corollary}

\theoremstyle{definition}

\newtheorem*{ackn}{Acknowledgments}
\newtheorem{rema}[theo]{Remark}

\newcommand{\Spec}{\mathop{\mathrm{Spec}}\nolimits}
\newcommand{\tdeg}{\mathop{\mathrm{tr.deg}}\nolimits}
\newcommand{\Mat}{\mathop{\mathrm{Mat}}\nolimits}
\newcommand{\GL}{\mathop{\mathrm{GL}}\nolimits}
\newcommand{\Lie}{\mathop{\mathrm{Lie}}\nolimits}
\newcommand{\Aut}{\mathop{\mathrm{Aut}}\nolimits}
\newcommand{\Gal}{\mathop{\mathrm{Gal}}\nolimits}

\newcommand{\mbf}{\mathbf{f}}
\newcommand{\mbm}{\mathbf{m}}
\newcommand{\mcZb}{\overline{\mathcal{Z}}}

\newcommand{\Z}{\mathbb{Z}}
\newcommand{\Q}{\mathbb{Q}}
\newcommand{\LL}{\mathbb{L}}
\newcommand{\T}{\mathbb{T}}
\newcommand{\E}{\mathbb{E}}
\newcommand{\Fq}{\mathbb{F}_{q}}
\newcommand{\Gm}{\mathop{\mathbb{G}_{m}}\nolimits}
\newcommand{\Ga}{\mathop{\mathbb{G}_{a}}\nolimits}

\newcommand{\kinf}{k_{\infty}}
\newcommand{\kinfb}{\overline{k_{\infty}}}
\newcommand{\cinf}{\mathbb{C}_{\infty}}
\newcommand{\kb}{\overline{k}}
\newcommand{\kbt}{\overline{k}^{\times}}

\newcommand{\Fb}{\overline{F}}
\newcommand{\Fs}{F^{\mathrm{sep}}}

\newcommand{\lan}{L_{\alpha,n}}
\newcommand{\laln}{L_{\alpha_{1},n}}
\newcommand{\larn}{L_{\alpha_{r},n}}

\newcommand{\aval}{|\alpha|_{\infty}}
\newcommand{\ajval}{|\alpha_{j}|_{\infty}}
\newcommand{\aub}{\underline{\alpha}}
\newcommand{\nub}{\underline{n}}

\newcommand{\zn}{\zeta(n)}
\newcommand{\znn}{\zeta(n,n)}
\newcommand{\zQ}{\zeta_{\mathbb{Q}}}

\newcommand{\laan}{L_{\alpha,\alpha,n,n}}

\newcommand{\az}{a_{0}}
\newcommand{\azt}{\widetilde{a_{0}}}
\newcommand{\cz}{c_{0}}

\newcommand{\vp}{\varphi}
\newcommand{\pit}{\widetilde{\pi}}
\newcommand{\Psit}{\widetilde{\Psi}}
\newcommand{\GaM}{\Gamma_{M}}
\newcommand{\GaP}{\Gamma_{\Psi}}
\newcommand{\nq}{|\theta|_{\infty}^{\frac{nq}{q-1}}}

\newcommand{\Ker}{\mathop{\mathrm{Ker}}\nolimits}

\newcommand{\Cn}{C^{\otimes n}}

\title{Algebraic independence of the Carlitz period and the positive characteristic multizeta values at $n$ and $(n,n)$}
\author{Yoshinori Mishiba\thanks{
Graduate School of Mathematics, Kyushu University, 744, Motooka, Nishi-ku, Fukuoka, 819-0395, JAPAN.
\endgraf e-mail: \texttt{y-mishiba@math.kyushu-u.ac.jp}
\endgraf Supported by the JSPS Research Fellowships for Young Scientists.
}}
\date{}

\begin{document}
\maketitle

\begin{abstract}
Let $k$ be the rational function field over the finite field of $q$ elements and $\kb$ its fixed algebraic closure.
In this paper, we study algebraic relations over $\kb$ among the fundamental period $\pit$ of the Carlitz module
and the positive characteristic multizeta values $\zn$ and $\znn$ for an ``odd'' integer $n$,
where we say that $n$ is ``odd'' if $q-1$ does not divide $n$.
We prove that either they are algebraically independent over $\kb$ or satisfy some simple relation over $k$.
We also prove that if $2n$ is ``odd'' then they are algebraically independent over $\kb$.
\end{abstract}

\section{Introduction}
Let $k := \Fq(\theta)$ be the rational function field over the finite field of $q$ elements with variable $\theta$,
$p$ the characteristic of $k$,
$\kinf := \Fq(\!(\theta^{-1})\!)$ the $\infty$-adic completion of $k$,
$\kinfb$ a fixed algebraic closure of $\kinf$,
and $\kb$ the algebraic closure of $k$ in $\kinfb$.
Thakur (\cite{Tha1}) defined the positive characteristic multizeta values (MZV) by
$$\zeta(n_{1}, \dots, n_{d}) := \sum \frac{1}{a_{1}^{n_{1}} \cdots a_{d}^{n_{d}}} \in \kinf$$
for integers $n_{1}, \dots, n_{d} \geq 1$, where the sum is over all monic polynomials $a_{i}$ in $\Fq[\theta]$ such that
$\deg a_{1} > \cdots > \deg a_{d} \geq 0$.
Throughout this paper, we use the notation $\zeta$ for the above means,
and we use the notation $\zQ$ for the classical multiple zeta values in characteristic zero.
The sum $\sum_{i} n_{i}$ is called the weight and $d$ is called the depth of the MZV $\zeta(n_{1}, \dots, n_{d})$.
We fix a ($q-1$)-st root of $-\theta$ and set
$$\pit := (-\theta)^{\frac{q}{q-1}} \prod_{i=1}^{\infty}\left(1-\theta^{1-q^i}\right)^{-1} \in (-\theta)^{\frac{1}{q-1}} \cdot \kinf^{\times}$$
the fundamental period of the Carlitz module.
Since $\# \Fq[\theta]^{\times} = q-1$, we say that an integer $n$ is ``odd'' if $q-1$ does not divide $n$, and ``even'' if $q-1$ divides $n$.

It is clear that $\zeta(p n) = \zn^{p}$ for all $n \geq 1$.
Carlitz (\cite{Carl}) showed that $\zn / \pit^{n} \in k^{\times}$ for each ``even'' integer $n \geq 1$.
Chang and Yu (\cite{ChYu}) proved that the all algebraic relations over $\kb$ among $\pit$ and MZVs of depth one come from the above types.
Thus we are interested in the case where the depth is greater than $1$.
In \cite{Chan}, Chang showed that the all algebraic relations over $\kb$ among MZVs are $k$-homogeneous.
However he did not treat linear relations over $\kb$ among MZVs of same weights.
In this paper, we prove the following theorem:

\begin{theo}\label{main}
Let $n \geq 1$ be an ``odd'' integer.
Then $\pit$, $\zn$ and $\znn$ are algebraically independent over $\kb$,
or $\zn^{2} - 2 \znn \in \pit^{2n} \cdot k^{\times}$.
If $2n$ is ``odd'', then we have the former case.
\end{theo}

\begin{rema}
If $p = 2$, then $2n$ is ``odd'' if and only if $n$ is ``odd''.
Thus $\pit$, $\zn$ and $\znn$ are algebraically independent over $\kb$ for each ``odd'' integer $n$.

On the other hand, in characteristic zero, $2n$ is always even.
Thus the second part of Theorem \ref{main} does not occur in this case.
In fact, we have the relation
$\zeta_{\Q}(n)^{2} - 2 \zeta_{\Q}(n,n) = \zeta_{\Q}(2n) \in \pi^{2n} \cdot \Q^{\times}$.
\end{rema}

\begin{rema}\label{shuffle}
If $p^{e}$ divides $n_{1}$ and $n_{2}$ and $n_{1} / p^{e} + n_{2} / p^{e} \leq q$ for some $e \geq 0$, then we have the harmonic shuffle product
$\zeta(n_{1}) \zeta(n_{2}) = \zeta(n_{1}, n_{2}) + \zeta(n_{2}, n_{1}) + \zeta(n_{1} + n_{2})$ (\cite[Theorem 1]{Tha2}, or see Remark \ref{H=1}).
In particular, if $2n = p^{e}(q-1)$ then we have the relation
$\zn^{2} - 2\znn = \zeta(2n) \in \pit^{2n} \cdot k^{\times}$
(when $p = 2$, this follows directly, but in this case $n$ is ``even'').
Thus, the latter case of the first part of Theorem \ref{main} actually occurs when $p \geq 3$.
We do not know what happens in the case where $2n = m (q-1)$ for general $m$ (including the case where $n$ is ``even'').
\end{rema}

Since $\pit$ and $\zn$ are algebraically independent over $\kb$ for each ``odd'' integer $n$ (\cite{ChYu}),
we have the following corollary:
\begin{coro}\label{pit-znn}
Let $n \geq 1$ be an ``odd'' integer.
Then $\pit$ and $\znn$ are algebraically independent over $\kb$.
\end{coro}
Corollary \ref{pit-znn} also follows from the ``Eulerian'' criterion (\cite{CPY}) and the fact that
if a multizeta value is not ``Eulerian'' then it is algebraically independent from $\pit$ over $\kb$ (\cite{Chan}).

\begin{coro}\label{Z2}
Let $\mcZb_{2}$ be the $\kb$-vector space spanned by $\zeta(1,1)$ and $\zeta(2)$
(the weight two MZVs).
We have $\dim_{\kb} \mcZb_{2} = 2$ if $q > 2$ and $\dim_{\kb} \mcZb_{2} = 1$ if $q = 2$.
\end{coro}

\begin{proof}
Note that by Remark \ref{shuffle}, we have $\zeta(1)^{2} = 2 \zeta(1,1) + \zeta(2) \in \mcZb_{2}$ for each $q$.
If $q \geq 4$ then $2$ is ``odd''.
Thus $\zeta(1)$ and $\zeta(1,1)$ (and $\pit$) are algebraically independent over $\kb$ by Theorem \ref{main}.
Thus $\zeta(1)^{2}$ and $\zeta(1,1)$ form a basis of $\mcZb_{2}$.
If $q = 3$ then $2$ is ``even'', and hence we have $\zeta(2) \in \pit^{2} \cdot k^{\times}$.
However $\pit$ and $\zeta(1)$ are algebraically independent over $\kb$ (\cite{ChYu}).
Thus $\zeta(1)^{2}$ and $\zeta(2)$ form a basis of $\mcZb_{2}$.
When $q = 2$, we have the relation $\zeta(1,1) = \zeta(2) / (\theta^{2} + \theta)$ (\cite[Theorem 5.10.13]{Tha1}).
\end{proof}

\begin{rema}
If $p \neq 2$ then $\zeta(1)$ and $\zeta(2)$ are algebraically independent over $\kb$ (\cite{ChYu}).
Thus a new result in Corollary \ref{Z2} is the characteristic $2$ case with $q \neq 2$.
\end{rema}

In Section \ref{t-motive}, we review Papanikolas' theory for $t$-motives briefly.
Theorem \ref{main} is proved in Section \ref{proof}.
The outline of the proof is as follows:
Following Anderson-Thakur (\cite{AT09}), we construct a rigid analytically trivial pre-$t$-motive $M$ such that the MZVs we want to know appear in its ``period matrix'' $\Psi$.
Then we obtain two group schemes $\GaM$ and $\GaP$,
where $\GaM$ is the fundamental group of the Tannakian category generated by $M$,
and $\GaP$ is defined from $\Psi$ more explicitly than $\GaM$.
By Papanikolas' theory, these two group schemes are isomorphic and their dimensions are the same as the transcendental degree which we want,
in this case, the transcendental degree of $\kb(\pit, \zn, \znn)$ over $\kb$.
By the Tannakian duality, we have a natural projection $\GaM \to \Gm$.
Set $V$ to be its kernel.
Since $\pit$ and $\zn$ are algebraically independent over $\kb$, the dimension of $\GaM$ is greater than or equal to $2$.
We suppose that the dimension of $\GaM$ is $2$.
Then we can determine $V$ explicitly.
By using the expression of $V$, we can also compute $\GaM$ explicitly.
Then we obtain a non-trivial defining polynomial of $\GaM$.
By returning to the definition of $\GaP$,
the MZVs in question must satisfy some relation coming from the defining polynomial.
If $2n$ is ``odd'', we can check easily that the MZVs do not satisfy this relation, and so we have a contradiction,
showing that $\tdeg_{\kb} \kb(\pit, \zn, \znn) \geq 3$.

\begin{ackn}
The author would like to thank Yuichiro Taguchi for reading a preliminary version of this paper carefully
and giving him many useful comments.
He also would like to thank
Chieh-Yu Chang for many comments and suggestions about the treatment of smoothness of group schemes in the proof of Theorem \ref{L-main},
and Dinesh S.\ Thakur for informing him of some relations among multizeta values in Remark \ref{shuffle} and Corollary \ref{Z2}.
This work was supported by the JSPS Research Fellowships for Young Scientists.
\end{ackn}

\section{Review of $t$-motives}\label{t-motive}
In this section, we review the notions of pre-$t$-motives.
For more details, see \cite{Papa}.
We continue to use the notations of the previous section.
Let $\cinf$ be the $\infty$-adic completion of $\kinfb$ and $|\cdot|_{\infty}$ its multiplicative valuation.
For any subset $T$ of $\cinf$, we denote by $\langle T \rangle_{k}$ the $k$-vector subspace of $\cinf$ spanned by $T$.
Let $t$ be a variable independent of $\theta$
and $\cinf(\!(t)\!)$ the field of formal Laurent series over $\cinf$.
Let $\T := \{f \in \cinf[\![t]\!] | f \mathrm{ \ converges \ on \ } |t|_{\infty} \leq 1\}$ be the Tate algebra
and $\LL$ the fractional field of $\T$.
We set
$$\E := \left\{\sum a_{i} t^{i} \in \cinf[\![t]\!] \Bigl| \lim_{i \to \infty} \sqrt[i]{|a_{i}|_{\infty}} = 0, \ [\kinf(a_{0}, a_{1}, \dots):\kinf] < \infty \right\}.$$
For any integer $n \in \Z$ and any formal Laurent series $f = \sum_{i} a_{i} t^{i} \in \cinf(\!(t)\!)$, we set
$$ f^{(n)} := \sum_{i} a_{i}^{q^n} t^{i} $$
the $n$-fold twist of $f$, and $\sigma(f) := f^{(-1)}$.
The operation $f \mapsto f^{(n)}$ stabilizes the subfields $\LL$ and $\kb(t)$ of $\cinf(\!(t)\!)$.
For a ring or a module $R$, we denote by $\Mat_{r \times s}(R)$ the set of $r \times s$ matrices with coefficients in $R$.

A pre-$t$-motive is an \'etale $\vp$-module over $(\kb(t), \sigma)$,
it means a finite-dimensional $\kb(t)$-vector space $M$ equipped with a $\sigma$-semilinear bijective map
$\vp \colon M \to M$.
A morphism of pre-$t$-motives is a $\kb(t)$-linear map which is compatible with the $\vp$'s.
There is a tensor product of two pre-$t$-motives.
For any pre-$t$-motive $M$, the Betti realization of $M$ is an $\Fq(t)$-vector space defined by
$$M^{B} := \left(\LL \otimes_{\kb(t)} M \right)^{\sigma \otimes \vp = 1},$$
where $(-)^{\sigma \otimes \vp = 1}$ is the $\sigma \otimes \vp$-fixed part.
Then we have a natural injection
$$\LL \otimes_{\Fq(t)} M^{B} \hookrightarrow \LL \otimes_{\kb(t)} M.$$
A pre-$t$-motive $M$ is called {\it rigid analytically trivial} if the above injection is an isomorphism.
The category of rigid analytically trivial pre-$t$-motives forms a neutral Tannakian category over $\Fq(t)$
with fiber functor $M \mapsto M^{B}$.
For any such $M$, we set $\GaM$ to be the fundamental group of the Tannakian subcategory generated by $M$ with respect to the Betti realization.
By definition, $\GaM$ is an $\Fq(t)$-subgroup scheme of $\GL(M^{B})$.

Let $M$ be a pre-$t$-motive and $r$ the dimension of $M$ over $\kb(t)$.
For a $\kb(t)$-basis $\mbm \in \Mat_{r\times 1}(M)$ of $M$,
there exists a unique matrix $\Phi \in \GL_{r}(\kb(t))$ such that $\vp \mbm = \Phi \mbm$.
Conversely when a matrix $\Phi \in \GL_{r}(\kb(t))$ is given,
we can construct a pre-$t$-motive $M$ by $M := \kb(t)^{r}$ with $\vp (x_{1}, \dots, x_{r}) := (x_{1}^{(-1)}, \dots, x_{r}^{(-1)}) \Phi$.
Clearly the pre-$t$-motive $M$ is determined by the matrix $\Phi$,
and we say that $M$ is the pre-$t$-motive defined by $\Phi$.
We can show that $M$ is rigid analytically trivial if and only if
there exists a matrix $\Psi \in \GL_{r}(\LL)$ such that $\Psi^{(-1)} = \Phi \Psi$.
The matrix $\Psi$ is called a rigid analytic trivialization of $\Phi$.
Let $\Psi_{1}, \Psi_{2} \in \GL_{r}(\LL \otimes_{\kb(t)} \LL)$ be the matrices such that
$(\Psi_{1})_{ij} = \Psi_{ij} \otimes 1$ and $(\Psi_{2})_{ij} = 1 \otimes \Psi_{ij}$
and set $\tilde{\Psi} := \Psi_{1}^{-1} \Psi_{2}$.
Let $X := (X_{ij})_{ij}$ be the $r \times r$ matrix of independent variables $X_{ij}$.
We define an $\Fq(t)$-algebra homomorphism $\nu$ by
$$\nu \colon \Fq(t)[X_{11}, X_{12}, \dots, X_{rr}, 1/\det X] \to \LL \otimes_{\kb(t)} \LL ; \ X_{ij} \mapsto \tilde{\Psi}_{ij}$$
and set
$$\GaP := \Spec ( \Fq(t)[X_{11}, X_{12}, \dots, X_{rr}, 1/\det X] / \Ker \nu ) \subset \GL_{r}.$$
We can easily check that $\Psi^{-1} \mbm$ is an $\Fq(t)$-basis of $M^{B}$.
For each $\Fq(t)$-algebra $R$, we have a map
$$\GaP(R) \to \Gamma_{M}(R); \ \gamma \mapsto (\mbf \cdot \Psi^{-1} \mbm \mapsto \mbf \gamma^{-1} \cdot \Psi^{-1} \mbm)$$
where $\mbf$ runs over all elements of $\Mat_{1 \times r}(R)$.

\begin{theo}[{\cite[Theorem 4.3.1, 4.5.10, 5.2.2]{Papa}}]\label{dim=tdeg}
Let $M$ be a rigid analytically trivial pre-$t$-motive equipped with $\Phi$ and $\Psi$ as above.

$(1)$ The scheme $\GaP$ is a smooth subgroup scheme of $\GL_{r}$ and 
the above map $\GaP \to \GaM$ is an isomorphism of group schemes over $\Fq(t)$.

$(2)$ We have
$$\dim \GaP = \tdeg_{\kb(t)} \kb(t)(\Psi_{11}, \Psi_{12}, \dots, \Psi_{rr}).$$

$(3)$Assume that $\Phi \in \Mat_{r \times r}(\kb[t])$, $\Psi \in \Mat_{r \times r}(\T)$,
and $\det \Phi = c (t-\theta)^{d}$ for some $c \in \kbt$ and $d \geq 0$.
Then we have $\Psi \in \Mat_{r \times r}(\E)$ and
$$\tdeg_{\kb(t)} \kb(t)(\Psi_{11}, \Psi_{12}, \dots, \Psi_{rr}) = \tdeg_{\kb} \kb(\Psi_{11}(\theta), \Psi_{12}(\theta), \dots, \Psi_{rr}(\theta)).$$
\end{theo}

When a $t$-motive $M$ has some special form, we can describe $\dim \GaM$ in terms of the dimension of some $k$-linear space.
We set
$$\Omega(t) := (-\theta)^{-\frac{q}{q-1}} \prod_{i=1}^{\infty}\left(1-\frac{t}{\theta^{q^i}}\right) \in \kinfb[\![t]\!],$$
which is an element of $\E$.
Since $\Omega$ has infinitely many zeros, it is transcendental over $\kb(t)$.
It satisfies the equations
$$\Omega^{(-1)} = (t-\theta) \Omega \ \ \mathrm{and} \ \ \Omega(\theta) = \frac{1}{\pit}.$$
For each $\alpha = \sum_{i} a_{i} t^{i} \in \kb[t]$, we set $\aval := \max_{i} \{ |a_{i}|_{\infty} \}$.
When $\aub = (\alpha_{1}, \dots, \alpha_{d})$ and $\nub = (n_{1}, \dots, n_{d})$ ($\alpha_{j} \in \kb[t] \smallsetminus \{ 0 \}$, $n_{j} \geq 1$)
satisfy $\ajval < |\theta|_{\infty}^{\frac{n_{j} q}{q-1}}$, we set
$$L_{\aub, \nub} (t) := \sum_{i_{1} > \cdots > i_{d} \geq 0}
\frac{\alpha_{1}^{(i_{1})} \cdots \alpha_{d}^{(i_{d})}}{((t-\theta^{q}) \cdots (t-\theta^{q^{i_{1}}}))^{n_{1}} \cdots ((t-\theta^{q}) \cdots (t-\theta^{q^{i_{d}}}))^{n_{d}}}.$$
Note that it converges on $|t|_{\infty} < |\theta|_{\infty}^{q}$ and satisfies
$$L_{\aub, \nub}^{(-1)} =
\frac{\alpha_{d}^{(-1)}}{(t-\theta)^{n_{1} + \cdots + n_{d-1}}} L_{\alpha_{1}, \dots, \alpha_{d-1}, n_{1}, \dots, n_{d-1}}
+ \frac{1}{(t-\theta)^{n_{1} + \cdots + n_{d}}} L_{\aub, \nub},$$
where we set $L_{\emptyset, \emptyset} := 1$ for the cases $d = 1$.
Note that when $\alpha_{i} \in \kbt$ for each $i$, $L_{\aub, \nub}(\theta)$ is the Carlitz multiple polylogarithm at the algebraic point $\aub$.

Let $n \geq 1$ be an integer and $\alpha_{1}, \dots, \alpha_{r} \in \kb[t] \smallsetminus \{ 0 \}$ with $|\alpha_{i}|_{\infty} < \nq$ for each $i$.
We set
$$\Phi :=
\begin{bmatrix} (t-\theta)^{n} & & & \\
\alpha_{1}^{(-1)} (t-\theta)^{n} & 1 & & \\
\vdots & & \ddots & \\
\alpha_{r}^{(-1)} (t-\theta)^{n} & & & 1 \\ \end{bmatrix} \in \GL_{r+1}(\kb(t)),$$

$$\Psi :=
\begin{bmatrix} \Omega^{n} & & & \\
\Omega^{n} \laln & 1 & & \\
\vdots & & \ddots & \\
\Omega^{n} \larn & & & 1 \\ \end{bmatrix} \in \GL_{r+1}(\LL).$$
Then we have
$$\Psi^{(-1)} = \Phi \Psi.$$

\begin{theo}[{\cite[Theorem 6.3.2]{Papa}}, {\cite[Theorem 3.1]{ChYu}}]\label{lin-alg}
Let $M$ be the $t$-motive defined by the matrix $\Phi$ as above.
Then we have
$$\dim \GaM = \dim_{k} \langle \pit^{n}, \laln(\theta), \dots, \larn(\theta) \rangle_{k}.$$
\end{theo}

\begin{rema}
In \cite{Papa} and \cite{ChYu}, they discussed only the case where $\alpha_{i} \in \kbt$.
But their proofs work also for any $\alpha_{i} \in \kb[t] \smallsetminus \{ 0 \}$.
\end{rema}

\section{Proof of Theorem \ref{main}}\label{proof}
We define two square matrices of size $d + 1$
$$\Phi :=
\begin{bmatrix} (t-\theta)^{n_{1} + \cdots + n_{d}} & & & & \\
\alpha_{1}^{(-1)} (t-\theta)^{n_{1} + \cdots + n_{d}} & (t-\theta)^{n_{2} + \cdots + n_{d}} & & & \\
& \alpha_{2}^{(-1)} (t-\theta)^{n_{2} + \cdots + n_{d}} & \ddots & & \\
& & \ddots & (t-\theta)^{n_{d}} & \\
& & & \alpha_{d}^{(-1)} (t-\theta)^{n_{d}} & 1 \\ \end{bmatrix},$$
and $\Psi := (\Psi_{ij})$, where
$$\Psi_{ij} := \Omega^{n_{j} + \cdots + n_{d}} L_{\alpha_{j}, \dots, \alpha_{i-1}, n_{j}, \dots, n_{i-1}} \ \ \mathrm{if} \ \ 1 \leq j \leq i \leq d + 1,$$
and $\Psi_{ij} := 0$ if $i < j.$
Then we have $\Psi^{(-1)} = \Phi \Psi$.
In particular each component $\Psi_{ij}$ of $\Psi$ is an element of $\E$ by Theorem \ref{dim=tdeg}.

We set $D_{i} := \prod_{j=0}^{i-1} (\theta^{q^{i}} - \theta^{q^{j}})$ for $i \geq 0$.
For each integer $n \geq 0$ with the $q$-expansion $n = \sum_{i} n_{i} q^{i}$ ($0 \leq n_{i} < q$),
the Carlitz factorial is defined by
$$\Gamma_{n+1} := \prod_{i} D_{i}^{n_{i}}.$$
We consider the power sum $S_{i}(n) := \sum\frac{1}{a^{n}}$,
where the sum is over all monic polynomials $a$ in $\Fq[\theta]$ with $\deg a = i$.
For each $n \geq 1$, Anderson and Thakur (\cite{AT90}, \cite{AT09}) defined polynomial $H_{n-1} \in \Fq[\theta, t]$ such that
$$(H_{n-1} \Omega^{n})^{(i)}(\theta) = \frac{\Gamma_{n} S_{i}(n)}{\pit^{n}}$$
for each $i \geq 0$ and $|H_{n-1}|_{\infty} < \nq$.
Chang (\cite{Chan}) showed that
$$(\Omega^{n_{1} + \cdots + n_{d}} L_{H(\nub), \nub})(\theta^{q^{N}})
= \left ( \frac{\Gamma_{n_{1}} \cdots \Gamma_{n_{d}} \zeta(n_{1}, \dots, n_{d})}{\pit^{n_{1} + \cdots + n_{d}}} \right )^{q^{N}}$$
for each $N \geq 0$,
where $\nub := (n_{1}, \dots, n_{d})$ and $H(\nub) := (H_{n_{1}-1}, \dots, H_{n_{d}-1})$.
Thus to prove Theorem \ref{main}, we shall consider the algebraic relations between the special values of $\Omega$, $\lan$ and $\laan$.

\begin{rema}\label{H=1}
We can easily show that
$$L_{\alpha_{1}, n_{1}} L_{\alpha_{2}, n_{2}} = L_{\alpha_{1}, \alpha_{2}, n_{1}, n_{2}} + L_{\alpha_{2}, \alpha_{1}, n_{2}, n_{1}} + L_{\alpha_{1} \alpha_{2}, n_{1} + n_{2}}$$
for each $\alpha_{i}$ and $n_{i}$
(for more general cases, see \cite[Section 6.2]{Chan}. He treated $L_{\aub, \nub}(\theta)$, but the arguments are the same).
By definition, $\Gamma_{n} = 1$ for each $1 \leq n \leq q$,
and by the construction in \cite{AT90}, we know that $H_{n-1} = 1$ for $1 \leq n \leq q$.
Thus if $n_{1} + n_{2} \leq q$, we have
$$L_{H_{n_{1} - 1} H_{n_{2} - 1}, n_{1} + n_{2}} = L_{1, n_{1} + n_{2}} = L_{H_{n_{1} + n_{2} - 1}, n_{1} + n_{2}}.$$
Therefore, we obtain the harmonic shuffle product formula in Remark \ref{shuffle}.
\end{rema}

We fix $\alpha \in \kb[t] \smallsetminus \{ 0 \}$ and $n \geq 1$ such that $\aval < \nq$,
and set matrices
$$\Phi :=
\begin{bmatrix} (t-\theta)^{2n} & & \\
\alpha^{(-1)} (t-\theta)^{2n} & (t-\theta)^{n} & \\
& \alpha^{(-1)} (t-\theta)^{n} & 1 \\ \end{bmatrix} \in \GL_{3}(\kb(t))$$
and
$$\Psi :=
\begin{bmatrix} \Omega^{2n} & & \\
\Omega^{2n} \lan & \Omega^{n} & \\
\Omega^{2n} \laan & \Omega^{n} \lan & 1 \\ \end{bmatrix} \in \GL_{3}(\LL).$$
These matrices satisfy $\Psi^{(-1)} = \Phi \Psi$.
Let $M$ be the rigid analytically trivial pre-$t$-motive defined by $\Phi$ and set $\Gamma := \GaM$ its fundamental group.
By Theorem \ref{dim=tdeg}, we have
$$\dim \Gamma = \tdeg_{\kb(t)} \kb(t)(\Omega, \lan, \laan) = \tdeg_{\kb} \kb(\pit, \lan(\theta), \laan(\theta)).$$

\begin{theo}\label{L-main}
Let $\alpha$ and $n$ be as above.
Assume that $\pit^{n}$ and $\lan(\theta)$ are linearly independent over $k$.
If $p > 2$, assume further that $(\Omega^{n} \lan)^{2} - 2 \Omega^{2n} \laan - c$ and $\Omega^{2n}$ are linearly independent over $\kb(t)$ for each $c \in \Fq(t)$.
Then we have
$$\tdeg_{\kb} \kb(\pit, \lan(\theta), \laan(\theta)) = 3.$$
\end{theo}

\begin{proof}
Let $\Fb$ be a fixed algebraic closure of $F := \Fq(t)$ and $\Fs$ the separable closure of $F$ in $\Fb$.
For a scheme $S$ over $F$, we write $S_{\Fb}$ for its base extension to $\Fb$.
We set
$$\Phi' :=
\begin{bmatrix} (t-\theta)^{n} & \\
\alpha^{(-1)} (t-\theta)^{n} & 1 \\ \end{bmatrix} \in \GL_{2}(\kb(t))$$
and
$$\Psi' :=
\begin{bmatrix} \Omega^{n} & \\
\Omega^{n} \lan & 1 \\ \end{bmatrix} \in \GL_{2}(\LL).$$
Let $M'$ be the pre-$t$-motive defined by $\Phi'$ and $\Gamma'$ its fundamental group.
Let $\Cn$ be the $n$-th tensor power of the Carlitz motive,
which is the pre-$t$-motive defined by the $1 \times 1$ matrix $\begin{bmatrix} (t-\theta)^{n} \end{bmatrix}$.
Clearly, the matrix $\begin{bmatrix} \Omega^{n} \end{bmatrix}$ is a rigid analytic trivialization of $\begin{bmatrix} (t-\theta)^{n} \end{bmatrix}$.
Then we have homomorphisms of pre-$t$-motives
$$M \twoheadrightarrow M'; \ (x_{1}, x_{2}, x_{3}) \mapsto (x_{2}, x_{3})$$
and
$$\Cn \hookrightarrow M'; \ x \mapsto (x, 0).$$
By the Tannakian duality and Theorem \ref{dim=tdeg}, we have a diagram of smooth group schemes over $F$
\[\xymatrix{
\Gamma \ar@{>>}[r]^{\psi} & \Gamma' \ar@{>>}[r]^{\pi'} & \Gamma_{\Cn} & \\
\GaP \ar[u]^{\simeq} & \Gamma_{\Psi'} \ar[u]^{\simeq} & \Gamma_{\Omega^{n}} \ar[u]^{\simeq} \ar@{=}[r] & \Gm. \\
}\]
In the following, we identify the upper group schemes with the lower group schemes in the above diagram.
At first, we describe the morphisms $\psi$ and $\pi'$ in the above diagram explicitly.
By the definition of $\Psit = (\Psit_{ij})_{ij}$, we have the relations
$\Psit_{11} = \Psit_{22}^{2}$, $\Psit_{22} = \Omega^{-n} \otimes \Omega^{n}$, $\Psit_{33} = 1$,
$\Psit_{21} = \Psit_{22} \Psit_{32}$, $\Psit_{32} = 1 \otimes \Omega^{n} \lan - \lan \otimes \Omega^{n}$,
$\Psit_{31} = (\lan^{2} - \laan) \otimes \Omega^{2n} - \lan \otimes \Omega^{2n} \lan + 1 \otimes \Omega^{2n} \laan$,
and $\Psit_{ij} = 0$ if $i < j$.
Thus we have the inclusion
$$\Gamma \cong \GaP \subset G := \left \{ \begin{bmatrix}
a^{2} & & \\
a x & a & \\
y & x & 1
\end{bmatrix} \right \},$$
where we use the letters $a, x, y$ as coordinate variables.
Similarly, we obtain
$$\Gamma' \cong \Gamma_{\Psi'} \subset G' := \left \{ \begin{bmatrix}
a & \\
x & 1
\end{bmatrix} \right \}.$$
Since $\dim \Gamma_{\Psi'} = \dim_{k} \langle \pit^{n}, \lan(\theta) \rangle_{k} = 2 = \dim G'$ and $G'$ is irreducible and reduced,
we have the equality $\Gamma_{\Psi'} = G'$.
By using the above identifications, we can write
$$\psi \colon \Gamma \to \Gamma'; \ \begin{bmatrix} a^2 & & \\ a x & a & \\ y & x & 1 \end{bmatrix} \mapsto \begin{bmatrix} a & \\ x & 1 \end{bmatrix}$$
and
$$\pi' \colon \Gamma' \to \Gm; \ \begin{bmatrix} a & \\ x & 1 \end{bmatrix} \mapsto a.$$
We set $\pi := \pi' \circ \psi$, $V := \Ker \pi$ and $V' := \Ker \pi'$.
Then we have
$$V \subset W := \left \{ \begin{bmatrix} 1 & & \\ x & 1 & \\ y & x & 1 \end{bmatrix} \right \}, \ \
V' = \left \{ \begin{bmatrix} 1 & \\ x & 1 \end{bmatrix} \right \} \cong \Ga$$
and obtain the following diagram
\[\xymatrix{
1 \ar[r] & W \ar[r] & G \ar[r] & \Gm \ar[r] \ar@{=}[d] & 1 \\
1 \ar[r] & V \ar@{^{(}->}[u] \ar[r] \ar[d]^{\psi|_{V}} & \Gamma \ar@{^{(}->}[u] \ar[r]^{\pi} \ar@{>>}[d]^{\psi} & \Gm \ar[r] \ar@{=}[d] & 1 \\
1 \ar[r] & V' \ar[r] & \Gamma' \ar[r]^{\pi'} & \Gm \ar[r] & 1 \\
}\]
which is commutative and whose rows are exact.
Clearly $\psi|_{V}$ is surjective.
The group scheme $V$ is smooth over $F$.
Indeed, take $\az \in \Gm(\Fb) \smallsetminus \Gm(\overline{\Fq})$ and its lift $\azt \in \Gamma(\Fb)$.
Let $T \subset \Gamma_{\Fb}$ be the Zariski closure of the group generated by $\azt$.
Then $T$ is a rank one torus and isomorphic to $\Gm_{,\Fb}$ via $\pi$.
In particular, $\Lie(\pi_{\Fb}) \colon \Lie(\Gamma_{\Fb}) \to \Lie(\Gm_{,\Fb})$ is nonzero.
Since $\Gamma_{\Fb}$ and $\Gm_{,\Fb}$ are smooth over $\Fb$,
we have $\dim V_{\Fb} = \dim \Lie(V_{\Fb})$,
and hence $V$ is smooth over $F$.

If $\dim \Gamma = 3$, then there is nothing to prove.
Thus we suppose that $\dim \Gamma = 2$ and hence $\dim V = 1$.
Then $\Ker (\psi|_{V} \colon V(\Fb) \to V'(\Fb)) = V(\Fb) \cap W_{1}(\Fb)$ has dimension zero, where we set
$$W_{1} := \left \{ \begin{bmatrix} 1 & & \\ & 1 & \\ y & & 1 \end{bmatrix} \right \}.$$
In view of the above short exact sequence, we let $\Gm(\Fb)$ act on $V(\Fb)$ by
$a.X := \widetilde{a}^{-1} X \widetilde{a}$ for $a \in \Gm(\Fb)$ and $X \in V(\Fb)$,
where $\widetilde{a} \in \Gamma(\Fb)$ is a lift of $a$.
In term of matrices, this action is given by
$$a . \begin{bmatrix} 1 & & \\ x & 1 & & \\ y & x & 1 \end{bmatrix} = \begin{bmatrix} 1 & & \\ a x & 1 & & \\ a^{2} y & a x & 1 \end{bmatrix}.$$
Since $V(\Fb) \cap W_{1}(\Fb)$ is closed under this action, we conclude that $V(\Fb) \cap W_{1}(\Fb) = 1$.
For a matrix
$$X = \begin{bmatrix} 1 & & \\ x & 1 & & \\ y & x & 1 \end{bmatrix} \in V(\Fb),$$
we have
$$X^{r} = \begin{bmatrix} 1 & & \\ r x & 1 & & \\ \frac{(r-1) r}{2} x^{2} + r y & r x & 1 \end{bmatrix} \in V(\Fb)$$
for each integer $r$.
Thus if $r$ is not divisible by $p$, we obtain
$$V(\Fb) \ni (r . X) X^{-r} = \begin{bmatrix} 1 & & \\ 0 & 1 & & \\ r((r-1) y - \frac{r-1}{2} x^{2}) & 0 & 1 \end{bmatrix} \in W_{1}(\Fb).$$
Since $V(\Fb) \cap W_{1}(\Fb) = 1$, we have the relation $$\frac{r-1}{2} x^{2} = (r-1) y.$$
When $p = 2$, if we take $r \equiv 3 \mod 4$, then we have $x^{2} = 2 y = 0$, and thus $x = 0$.
This means $X \in V(\Fb) \cap W_{1}(\Fb) = 1$, therefore $V(\Fb) = 1$.
This is a contradiction.
Thus when $p = 2$, we always have $\dim \Gamma = 3$.
In the following, we assume that $p \geq 3$.
In this case, if we take $r \not \equiv 0, 1 \mod p$, then we have the relation $y = \frac{x^{2}}{2}$.
Since $\dim V = 1$, we conclude that
$$V = \left \{ \begin{bmatrix} 1 & & \\ x & 1 & \\ \frac{x^{2}}{2} & x & 1 \end{bmatrix} \right \}.$$
Next, we determine the group scheme $\Gamma$.
Fix an element $\az \in \Gm(F)$ which has infinite order and set $\overline{\az} \in \Gm(\Fb)$ to be the geometric point above $\az$.
Since the geometric fiber $\Gamma_{\overline{\az}}$ of $\pi$ over $\overline{\az}$ is a $V_{\Fb}$-torsor,
it is isomorphic to $V_{\Fb}$ which is smooth over $\Fb$.
Thus the fiber $\Gamma_{\az}$ is smooth over $F$.
By \cite[Chapter 3, Proposition 2.20]{Liu}, we have $\Gamma_{\az}(\Fs) \neq \emptyset$, and hence we can take a lift $\azt$ of $\az$ in $\Gamma(\Fs)$.
Since $\Gamma(\Fs)$ contains $V(\Fs)$, we can eliminate the $x$-coordinate of $\azt$.
Thus we may assume that
$$\azt = \begin{bmatrix} \az^{2} & & \\ & \az & \\ y_{0} & & 1 \end{bmatrix} \in \Gamma(\Fs).$$
Then for each $r \in \Z$, we have
$$\azt^{r} = \begin{bmatrix} \az^{2 r} & & \\ & \az^{r} & \\ \frac{y_{0}}{\az^{2} - 1} (\az^{2 r} - 1) & & 1 \end{bmatrix} \in \Gamma(\Fs).$$
Since $\az$ has infinite order, we have
$$\left \{ \begin{bmatrix} a^{2} & & \\ & a & \\ \frac{\cz}{2}(a^{2} - 1) & & 1 \end{bmatrix} \right \} \subset \Gamma_{\Fb},$$
where $\cz := \frac{2 y_{0}}{\az^{2} - 1} \in \Fs$.
Since $\Gamma_{\Fb}$ is a two-dimensional irreducible reduced group scheme which also contains $V_{\Fb}$,
we conclude that
$$\Gamma_{\Fb} = \left \{ \begin{bmatrix} a^{2} & & \\ a x & a & \\ \frac{\cz}{2}(a^{2} - 1) + \frac{x^{2}}{2} & x & 1 \end{bmatrix} \right \}.$$
We set a polynomial
$$Q := 2 X_{31} - \cz(X_{22}^{2} - 1) - X_{32}^{2} \in \Fs[X_{22}, X_{31}, X_{32}] \subset \Fs[X_{11}, \dots, X_{33}, 1/\det X].$$
Then $\Gamma_{\Fb} = \Spec \Fb[X_{22}, X_{31}, X_{32}, X_{22}^{-1}] / (Q)$.
Since $\Gamma$ is defined over $F$, the ideal $(Q)$ is stable under the action of $\Gal(\Fs/F) = \Aut(\Fb/F)$.
Therefore for each $\sigma \in \Gal(\Fs/F)$, we can write $\sigma(Q) = P_{\sigma} Q$ for some $P_{\sigma} \in \Fb[X_{22}, X_{31}, X_{32}, X_{22}^{-1}]$.
By comparing the degree of each variables, $P_{\sigma}$ must be a constant.
Comparing the both sides again, we have $P_{\sigma} = 1$ and $\sigma(\cz) = \cz$ for each $\sigma \in \Gal(\Fs/F)$.
Hence we have $\cz \in F$ and $Q \in F[X_{11}, \dots, X_{33}, 1/\det X]$.
Since $Q \equiv 0$ on the reduced scheme $\Gamma_{\Fb}$, we have $Q(\Psit) = 0$.
By the definition of $\Psit$, this is equivalent to the equality
$$((\Omega^{n} \lan)^{2} - 2 \Omega^{2n} \laan - \cz) \otimes \Omega^{2n} = \Omega^{2n} \otimes ((\Omega^{n} \lan)^{2} - 2 \Omega^{2n} \laan - \cz)$$
in $\LL \otimes_{\kb(t)} \LL$.
However, this is a contradiction to the assumption about the linear independence.
\end{proof}

\begin{proof}[Proof of Theorem \ref{main}]
We fix an ``odd'' integer $n \geq 1$ and set $\alpha := H_{n-1}$.
Since $\pit^{n} \not \in \kinf$ and $\lan(\theta) = \Gamma_{n} \zeta(n) \in \kinf^{\times}$, they are linearly independent over $k$.
If $(\Omega^{n} \lan)^{2} - 2 \Omega^{2n} \laan - c$ and $\Omega^{2n}$ are linearly independent over $\kb(t)$ for each $c \in \Fq(t)$,
we have that $\pit$, $\zeta(n)$ and $\zeta(n,n)$ are algebraically independent over $\kb$ by Theorem \ref{L-main}.
Suppose that $(\Omega^{n} \lan)^{2} - 2 \Omega^{2n} \laan - c$ and $\Omega^{2n}$ are linearly dependent over $\kb(t)$ for some $c \in \Fq(t)$.
Then there exists an element $f \in \kb(t)$ such that $(\Omega^{n} \lan)^{2} - 2 \Omega^{2n} \laan - c = f \Omega^{2n}$.
Since $c$ and $f$ are rational functions, these have no poles nor zeros at $t = \theta^{q^{N}}$ for some $N \geq 0$.
Then we have
$$\left ( \frac{\Gamma_{n} \zeta(n)}{\pit^{n}} \right )^{2 q^{N}} - 2 \left ( \frac{\Gamma_{n} \Gamma_{n} \zeta(n,n)}{\pit^{2n}} \right )^{q^{N}} - c(\theta^{q^{N}}) = 0.$$
Thus we obtain the relation
$$\zeta(n)^{2} / \pit^{2n} - 2 \zeta(n,n) / \pit^{2n} = c(\theta) / \Gamma_{n}^{2} \in k^{\times}.$$

Now we assume that $2n$ is ``odd''.
Then the left hand side of the above equation is contained in $\pit^{-2n} \cdot \kinf$.
However $\pit^{2n} \not \in \kinf$ by the definition of $\pit$.
This is a contradiction.
Thus the linear independence holds in this case.
\end{proof}

\begin{rema}
In the proof of Theorem \ref{main}, we used Theorem \ref{lin-alg}
to show the algebraic independence of $\pit$ and $\zeta(n)$ over $\kb$ for an ``odd'' integer $n \geq 1$
(this is equivalent to $\dim \Gamma' = 2$ by Theorem \ref{dim=tdeg}).
However, by using similar arguments to determine the dimension of $\Gamma$, we can show that $\dim \Gamma' = 2$ directly.
This gives a simple proof of the algebraic independence of $\pit$ and $\zeta(n)$ over $\kb$, which is proved in \cite{ChYu}.
The proof is as follows:
Assume that $\dim \Gamma' = 1$.
Then $\dim V' = 0$, and hence $V' = 1$ since $\Gm$ acts on $V'$.
As before, we obtain
$$\Gamma' = \left \{ \begin{bmatrix} a & \\ \cz(a - 1) & 1 \end{bmatrix} \right \}$$
for some $\cz \in \Fq(t)$.
This implies the equation
$$(\Omega^{n} \lan + \cz) \otimes \Omega^{n} = \Omega^{n} \otimes (\Omega^{n} \lan + \cz)$$
in $\LL \otimes_{\kb(t)} \LL$.
Therefore if $\Omega^{n} \lan + c$ and $\Omega^{n}$ are linearly independent over $\kb(t)$ for each $c \in \Fq(t)$,
we have a contradiction, and hence $\dim \Gamma' = 2$.
When $q-1$ does not divide $n$ and $\alpha = H_{n-1}$, this condition is satisfied.
\end{rema}

\end{document}